\documentclass{amsart}

\usepackage{amsthm}
\usepackage{amsmath}
\usepackage{amssymb}
\usepackage{mathrsfs}

\newtheorem{theorem}{Theorem}[section]
\newtheorem{lemma}[theorem]{Lemma}

\theoremstyle{definition}
\newtheorem{definition}[theorem]{Definition}
\newtheorem{example}[theorem]{Example}

\theoremstyle{remark}

\newtheorem{conjecture}[theorem]{Conjecture}

\newtheorem{question}[theorem]{Question}
\usepackage{tikz-cd}
\usepackage{tikz}
\usetikzlibrary{positioning}
\usepackage{caption}
\usepackage{subcaption}
\usepackage{xcolor}
\usepackage{graphicx}

\numberwithin{equation}{section}

\begin{document}

\title{Generalized pseudo-Anosov maps and Hubbard trees}

\author{Mariam Al-Hawaj}
\address{Department of Mathematics, University of Toronto, Toronto, Ontario}
%\curraddr{Department of Mathematics,
%Case Western Reserve University, Cleveland, Ohio 43403}
\email{mariam.alhawaj@mail.utoronto.ca}

%    General info
%\subjclass[2020]{Primary 54C40, 14E20; Secondary 46E25, 20C20}

%\date{January 1, 2001 and, in revised form, June 22, 2001.}
\date{\today}

%\dedicatory{This paper is dedicated to my advisor.}

%\keywords{Differential geometry, algebraic geometry}

\begin{abstract}
In this project, we develop a new connection between the dynamics of quadratic polynomials on the complex plane and the dynamics of homeomorphisms of surfaces. In particular, given a quadratic polynomial, we investigate whether one can construct an extension of it which is a generalized pseudo-Anosov homeomorphism. \textit{Generalized pseudo-Anosov} means it preserves a pair of foliations with infinitely many singularities that accumulate on finitely many points. We determine for which quadratic polynomials such an extension exists. The construction is related to the dynamics on the Hubbard tree, which is a forward invariant subset of the filled Julia set containing the critical orbit. We define a type of Hubbard trees, which we call \textit{crossing-free}, and show that these are precisely the Hubbard trees for which one can construct a generalized pseudo-Anosov map.
            
\end{abstract}

\maketitle

%\section*{This is an unnumbered first-level section head}
%This is an example of an unnumbered first-level heading.

%% The correct journal style for \specialsection is all uppercase; a known bug
%% in amsart.cls prevents this, so input must be uppercase until it is fixed.
%\specialsection*{This is a Special Section Head}
%\specialsection*{THIS IS A SPECIAL SECTION HEAD}
%This is an example of a special section head%
%%%%%%%%%%%%%%%%%%%%%%%%%%%%%%%%%%%%%%%%%%%%%%%%%%%%%%%%%%%%%%%%%%%%%%%%
%\footnote{Here is an example of a footnote. Notice that this footnote
% text is running on so that it can stand as an example of how a footnote
% with separate paragraphs should be written.
%\par
% And here is the beginning of the second paragraph.}%
%%%%%%%%%%%%%%%%%%%%%%%%%%%%%%%%%%%%%%%%%%%%%%%%%%%%%%%%%%%%%%%%%%%%%%%%
.

\section{Introduction}
The Nielsen-Thurston classification \cite{FM} of mapping classes established that every orientation preserving homeomorphism of a closed surface, up to isotopy, is either periodic, reducible, or pseudo-Anosov. Pseudo-Anosov maps have a particularly nice structure because they expand along one foliation by a factor of $\lambda >1$ and contract along a transversal foliation by a factor of $\frac{1}{\lambda}.$ The number $\lambda $ is called the \textit{dilatation} of the pseudo-Anosov. Fried showed \cite{Fr} that every dilatation $\lambda$ of a pseudo-Anosov map is an algebraic unit, and conjectured that every algebraic unit $\lambda$ whose Galois conjugates lie in the annulus $A_\lambda =\{z: \frac{1}{\lambda} <|z|<\lambda\}$ is a dilatation of some pseudo-Anosov on some surface $S.$
       
       Pseudo-Anosovs play a huge role in Teichm\"{u}ller theory and geometric topology. The relation between these and complex dynamics has been well studied, inspired by Thurston.     

      The goal of this paper is to build a new connection between complex dynamics and Teichm$\ddot{\text{u}}$ller theory by constructing \textit{generalized pseudo-Anosov} maps of surfaces from polynomials acting on the complex plane.
      
       Generalized pseudo-Anosov maps are surface homeomorphisms that preserve a pair of transverse foliations where the foliations have infinitely many singularities that accumulate on finitely many points. Given a quadratic polynomial, we are interested in constructing an extension of it which is a generalized pseudo-Anosov homeomorphism.
       Recall that a polynomial is \textit{post-critically finite} if all its critical orbits are finite and it is \textit{superattracting} if all critical orbits are purely periodic.
       
 The construction will be related to the dynamics on the \textit{Hubbard tree} $T$, which is an invariant subset of the filled Julia set (see Section  \ref{Hubbard} for the definition). 
 The \textit{core entropy} of post-critically finite polynomial $f$ is the topological entropy of the restriction of $f$ to its Hubbard tree.
 
So basically, we are interested in the following question.   
\begin{question}
	If $f:T \to T$ is a post-critically finite quadratic polynomial with Hubbard tree $T$, is there a surface $S$ and a generalized pseudo-Anosov homeomorphism $\varphi : S \to S$ that is an extension of $f$? 

  \[
  \begin{tikzcd} 
    S\arrow{r}{\varphi} \arrow[swap]{d}{\pi} & S \arrow{d}{\pi} \\
      T \arrow{r}{f} & T
  \end{tikzcd}
\]

In other words, we require the above diagram to commute on an open dense subset of $S$. See Definition \ref{extension} for a formal definition. 
\end{question}   
\begin{definition}
Let $T$ be the Hubbard tree of a post-critically finite quadratic polynomial.  We say that $f$ is \textit{extendable} to a generalized pseudo-Anosov homeomorphism if there exist a surface $S$ and a generalized pseudo-Anosov $\varphi$ such that $\varphi$ is an extension of $f$.	
\end{definition}

 We will show that the type of Hubbard trees for which the construction is possible are those that are crossing-free and non-degenerate. 

 A Hubbard tree is \textit{non-degenerate} when the critical point is not an endpoint: thus, it divides the tree into two sub-trees. We call \textit{upper branch} ${B_u}$ the sub-tree containing the critical value, and \textit{lower branch} ${B_l}$ the other one. 
  
\begin{definition}
A Hubbard tree $T$ is said to be \textit{crossing-free} if there exists an embedding of $T$ into the plane $\mathbb{C}$ that satisfies the following: 

The image $f(B_l)$ of the lower branch can be isotoped, while fixing the endpoints of $T$, into a tree that does not intersect the interior of the image $f(B_u)$ of the upper branch.
  \end{definition}
 
Below are examples of both a crossing-free Hubbard tree and a Hubbard tree with crossing.
\begin{figure}[h]
    \centering
    \includegraphics[width=0.35\textwidth]{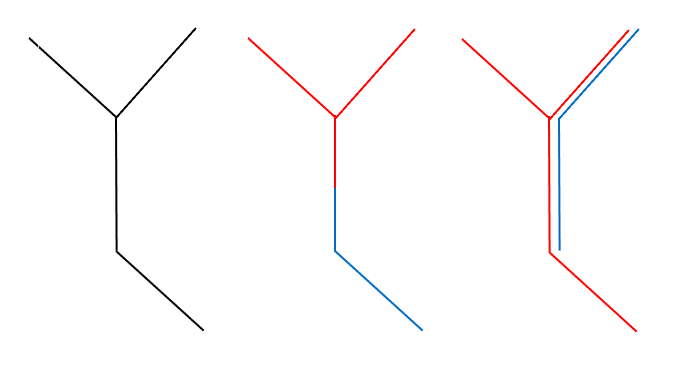}
    \caption{Crossing-free Hubbard tree.}
    \label{Crossing-free-example}
    \includegraphics[width=0.35\textwidth]{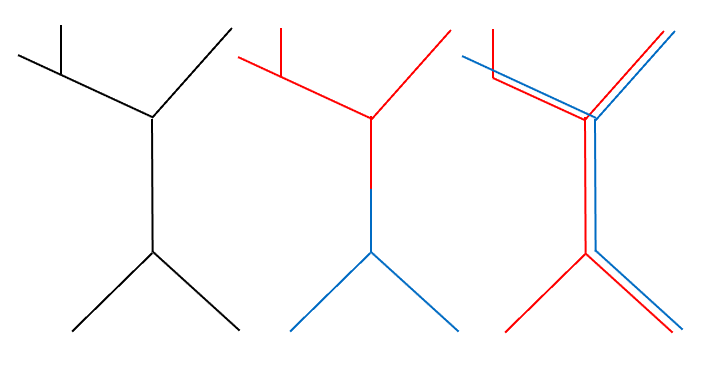}
    \caption{Hubbard tree with crossing.}
    \label{Crossing-example}
\end{figure}

   The main result of this paper is the following:  
 
     \begin{theorem} \label{theorem1}
     Let $f$ be a post-critically finite, superattracting quadratic polynomial and let $T$ be its Hubbard tree. 
 	  Then $f:T\to T$ is extendable to a generalized pseudo-Anosov $\varphi:S\to S$ if and only if it is crossing-free.
 	  Moreover, if $\lambda$ is the dilatation of $\varphi$, then $\log \lambda$ equals the core entropy of $f$.
 	  
 	    \end{theorem} 
 	  \begin{figure}[h]
     \centering
     \begin{subfigure}[b]{0.45\textwidth}
         \centering
         \includegraphics[width=\textwidth]{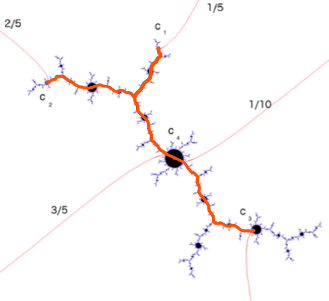}
         \caption*{Crossing-free Hubbard tree.}
     \end{subfigure}
     \hfill
     \begin{subfigure}[b]{0.45\textwidth}
         \centering
         \includegraphics[width=\textwidth]{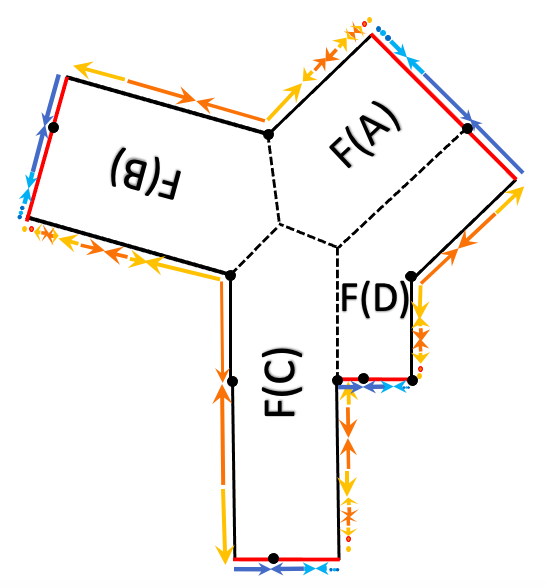}
         \caption*{Generalized pseudo-Anosov map.}
     \end{subfigure}
     \caption{The correspondence between the Hubbard tree and the generalized pseudo-Anosov map as in Theorem \ref{theorem1}.}
        \label{figure-3}
    \end{figure}
     
This result generalizes to complex polynomials the results of de Carvalho-Hall and Farber, who constructed generalized pseudo-Anosovs from real quadratic polynomials \cite{dCH} \cite{Fa}. 

To give an idea of the proof, the construction of the pseudo-Anosov map, that we will see in Section \ref{Main}, goes as follows: 
\begin{enumerate}
\item first, we thicken the tree and then we define a \textit{thick map} on it;
	\item since the map folds the thickened tree and to keep track of the folding, we construct a \textit{tree-like train track} by a procedure that's shown in Figure \ref{tree-like}; 
	\item we then use the dynamics of the train track to compute the transition matrix whose leading eigenvalue will be our dilatation $\lambda $;
	\item moreover, we use the eigenvectors of the matrix as the dimensions of the rectangles that we build;
	\item after that, we identify the rectangles to create the generalized pseudo-Anosov homeomorphism.
\end{enumerate}
To prove the converse (Theorem \ref{converse}), we use the invariant foliation of the generalized pseudo-Anosov to isotope apart the upper and the lower branches showing that the Hubbard tree must be crossing-free.
\subsection*{Acknowledgments}
The author would like to thank her advisor Giulio Tiozzo for bringing this problem to her attention and for his continuous help and support in writing this paper, as well as providing some of the complex dynamics pictures.
\section{Background}\label{Backgraound}
Let us recall some definitions.
%%%%%%%%%%%%%%%%%%%%%%%%%%%%%
%	Pseudo-Anosov Homeomorphisms
%%%%%%%%%%%%%%%%%%%%%%%%%%%%%
\subsection{Pseudo-Anosov Homeomorphisms}

\begin{definition}
 A homeomorphism $\varphi :S \rightarrow S$ is said to be \textit{pseudo-Anosov} is if there is a pair of transverse measured foliations $(\mathcal{F}^u,\mu^u)$ and $( \mathcal{F}^s ,\mu ^s )$ on $S$, and a number $\lambda >1$
 such that the following hold:      
    $$\varphi _*(\mathcal{F}^u,\mu^u) = (\mathcal{F}^u,\lambda \mu^u) \text { and } \ \varphi_*(\mathcal{F}^s,\mu^s) = (\mathcal{F}^s,\frac{1}{\lambda} \mu^s).$$
    The measured foliations $(\mathcal{F}^u,\mu^u)$ and $( \mathcal{F}^s ,\mu ^s )$ are called the \textit{unstable foliation}
and the \textit{stable foliation}, and the number $\lambda $ is called the
\textit{dilatation} of $\varphi$.    
    \end{definition}
% \begin{figure}[h]
% \centering
%  \includegraphics[scale=.3]{pseudoAnosov}
%  \caption{Pseudo-Anosov [ ]} 
%  \end{figure}

The following definition of generalized pseudo-Anosov map is taken from \cite{dCH}:
\begin{definition}
	A homeomorphism $\varphi :S\rightarrow S$ of a smooth surface $S$ is called a \textit{generalized pseudo-Anosov} map if there exist
	\begin{enumerate}
	\item a finite $\varphi -$invariant set $\Sigma$;
	\item a pair $(\mathcal{F}^u , \mu ^u) , \ (\mathcal{F}^s ,\mu ^s)$ of transverse measured foliations of $S\backslash \Sigma$ with \textit{countably many pronged singularities}, which accumulate on each point of $\Sigma $ and have no other accumulation points;
	\item  	a real number $\lambda > 1$; such that     
    $$\varphi (\mathcal{F}^u,\mu^u) = (\mathcal{F}^u,\lambda \mu^u) \text { and } \ \varphi (\mathcal{F}^s,\mu^s) = (\mathcal{F}^s,\frac{1}{\lambda} \mu^s).$$     
    \begin{figure}[h]
    \centering
    \includegraphics[width=.9\textwidth]{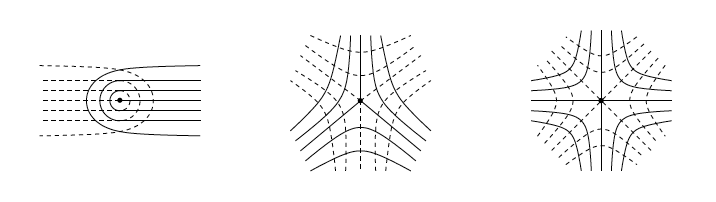}
    \caption{Pronged singularities of the invariant foliations.}
    \label{Pronged-sing}
\end{figure} 
	\end{enumerate}
\end{definition}   
%%%%%%%%%%%%%%%%%%%%%%%%%%%%%%%%%%
%   Complex Dynamics and Hubbard trees
%%%%%%%%%%%%%%%%%%%%%%%%%%%%%%%%%%
\subsection{Complex Dynamics and Hubbard trees}
Consider the family of quadratic polynomials $f_c(z):=z^2+c,$ where $c\in \mathbb {C}.$ For each parameter $c,\, f_c$ has a unique critical point $z=0.$ We recall the following definitions where we refer to \cite{DH} for more details:
\begin{definition}
	The \textit{filled Julia set $K(f_c)$} of $f_c$  is given by
$$K(f_c):=\{z\in \mathbb{C}: f^n_c (z) \kern.em\not\kern-.em\rightarrow \infty \}.$$
\end{definition}
  \begin{definition}
  	The \textit{Julia set $J(f_c)$} of $f_c$ is the boundary of the filled Julia set, i.e. $J(f_c):=\partial K(f_c).$
  \end{definition}
  \begin{definition}
  The \textit{Mandelbrot set $M$} is given by
$$M:=\{c\in \mathbb{C}: J(f_c)\  \text{is connected} \}.$$	
  \end{definition}
The complement of the Mandelbrot set is a non-empty, open set, and simply connected in the Riemann sphere $\hat {\mathbb{C}}.$ Consider the unique Riemann map $\Phi:\mathbb{C}\backslash \bar {\mathbb{D}} \rightarrow \mathbb{C}\backslash M$ with $\Phi (\infty)=\infty$ and $\Phi ^{'} (\infty) >0.$
We can define the external ray as follows:
\begin{definition}
 For $\theta \in \mathbb{R}/\mathbb{Z},$ we define the \textit{external ray at} $\theta$ to be the set $R_M(\theta):=\Phi (\{\rho e^{2\pi i\theta}, \ \rho >1\}).$
 An external ray $R_M(\theta)$ is said to \textit{land at $x$} if $$\underset {\rho \rightarrow 1^{+}} \lim {\Phi (\rho e^{2\pi i\theta})}=x.$$
\end{definition}
  \begin{figure}[h]
    \centering
    \includegraphics[width=.5\textwidth]{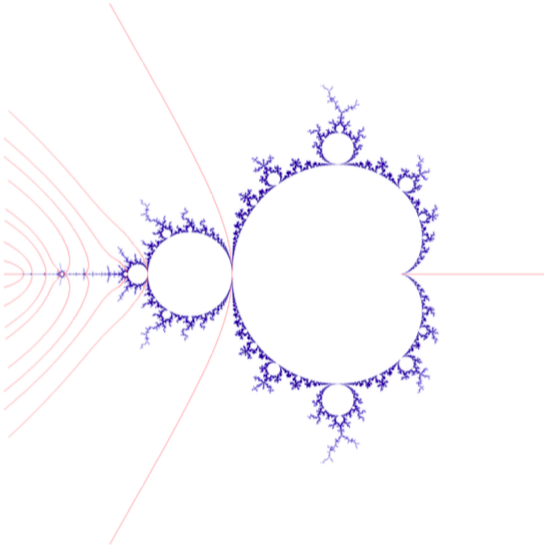}
    \caption{External rays.}
    \label{External-rays.}
\end{figure} 
\begin{definition}
A polynomial $f:\mathbb{C} \rightarrow \mathbb{C}$ is said to be \textit{post-critically finite} if the forward orbit of every critical point of $f$ is finite. 
\end{definition}
Recall a rational angle $\theta\in \mathbb{Q}/\mathbb{Z}$ determines a post-critically finite map $f_\theta(z) := z^2 + c_\theta.$ 
If $\theta$ is pre-periodic for the doubling map, then the external ray of angle $\theta$ lands at a post-critically finite parameter that we denote as $c_\theta$. 
If $\theta$ is purely periodic for the doubling map, then the external ray lands at the root of a hyperbolic component, and we let $c_\theta$ be the centre of such component.

\begin{definition}\label{Hubbard} 
Let $f$ be a post-critically finite quadratic polynomial with critical point $c_0$ and let $c_i:=f^i(c_0)$. We define the \textit{Hubbard tree} $T$ of $f$ \textit{Hubbard tree} of $f$ to be the union of the regulated arcs $$T=	\bigcup_{n,m \geq 0} [c_n,c_m].$$
It is a forward invariant subset of the filled Julia set $K(f)$ that contains the critical orbit.
\end{definition}

%%%%%%%%%%%%%%%%%%%%%%%%%%%%%
%	Example 1
%%%%%%%%%%%%%%%%%%%%%%%%%%%%%
\begin{example}\label{Example}
Let $f$ be a quadratic polynomial with the characteristic angle $\theta =\frac{1}{5}.$ Then the critical point $c_0$ has period $p=4.$ 
Then looking at the Julia set of $f,$ in Figure \ref{Hubtree}, we can see the Hubbard tree which is the union of the regulated arcs $[c_1,c_3]$ and $[c_2,c_3]$. Also we can see the external rays that land at the roots of the Fatou components containing the elements of the critical orbit. Note that the angles $\frac {\theta}{2}=\frac {1}{10}$ and $\frac {\theta +1}{2}=\frac {3}{5}$ both land on the boundary of the critical Fatou component containing the point $c_0=c_4$.
\begin{figure}[h]
\centering 
 \includegraphics[width=.5\textwidth]{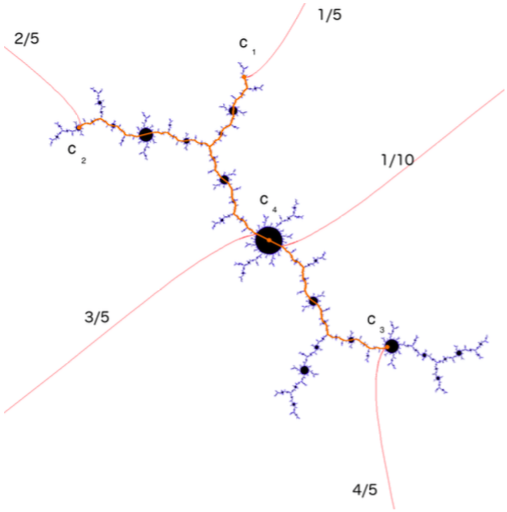}
 \caption{Hubbard tree for post-critically finite quadratic polynomial with angle $\theta =\frac{1}{5},$ and $p=4$.} 
  \label{Hubtree}
 \end{figure}
 \end{example}
%%%%%%%%%%%%%%%%%%%%%%%%%%%%%%%%%%%%%%%%%%%%%
%     Thick Hubbard Tree
%%%%%%%%%%%%%%%%%%%%%%%%%%%%%%%%%%%%%%%%%%%%%% 
  \subsection{Thickening the Hubbard Tree}
  \begin{definition}
  We define an \textit{$n$-star} $X$ to be a topological space homeomorphic to a disc with $2n$ marked points on the boundary denoted in counter-clockwise order as $p_i, i=1, \dots, 2n$. We define the \textit{inner boundary} $\partial_{in} X$ of $X$ as the subset of the boundary of $X$ given by the union of the arcs $$\partial_{in} X := \bigcup_{i=1}^n[p_{2i-1}, p_{2i}].$$ 
  Similarly, we define the \textit{outer boundary} $\partial_{out} X$ of $X$ as the union $$\partial_{out} X := \bigcup_{i=1}^{n}[p_{2i}, p_{2i+1}]$$ where $p_{2n+1}=p_1.$
  We call a $2$-star a \textit{rectangle}.
  \end{definition}
   \begin{figure}[h]
     \centering
     \begin{subfigure}[b]{0.45\textwidth}
         \centering
         \includegraphics[width=\textwidth]{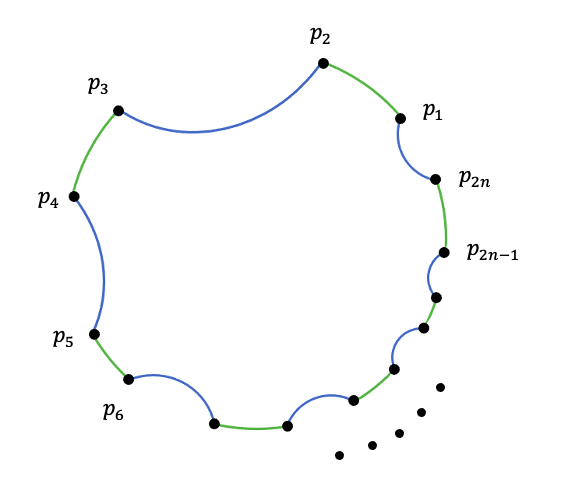}
         \caption*{$n$-star.}
     \end{subfigure}
     \hfill
     \begin{subfigure}[b]{0.4\textwidth}
         \centering
         \includegraphics[width=\textwidth]{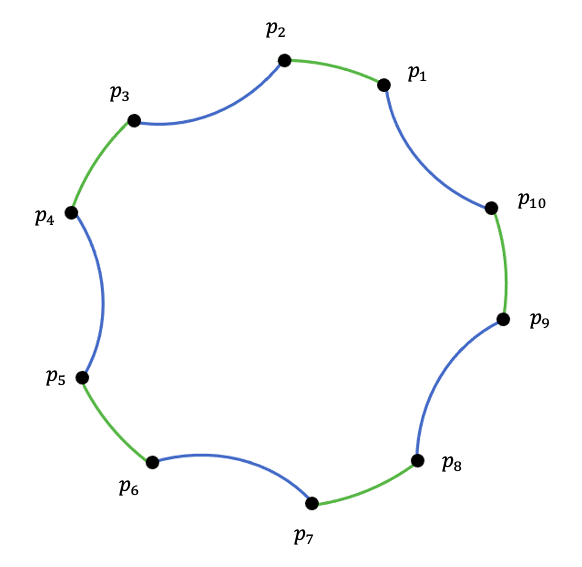}
         \caption*{$5$-star.}
     \end{subfigure}
     \caption{An $n$-star and a $5$-star with inner boundary in green and outer boundary in blue.}
        \label{n-star}
    \end{figure}

  We now give the definition of thick tree:
  \begin{definition}
 A \textit{thick tree} is a closed topological 2-disc $\mathbb{T}$ consisting of junctions, sides. 
 
  \textit{Junctions} are homeomorphic to a closed 2-disc and are of two types:
  \textit{end-junctions} and  \textit{inner-junctions} (representing the vertices and the post-critical set of the Hubbard tree, respectively).
 The boundary of the junctions are divided into two parts. One part intersects the boundary of the thick tree and the other part intersects the boundary of a side.  
 
\textit{Sides} are $n$-stars; moreover, every component of their inner boundary is contained in the boundary of a junction, and every component of their outer boundary is contained in the boundary of $\mathbb{T}$. For $n=2$, they are called \textit{simple sides} and represent the thickening of the edges of the Hubbard tree. For $n>2$, $n$-stars represent a thickening of a neighbourhood of an $n$-valence branch point in the Hubbard tree. 

 \end{definition} 
 \begin{definition}
 	 We say that the $n$-star $X$ \textit{connects} the junctions $J_1, \dots, J_k$ if the components of the inner boundary of $X$ are contained in the boundary of the junctions $J_1, \dots, J_k$.

 \end{definition}
\begin{figure}[h]
\centering
\includegraphics[width=.4\textwidth]{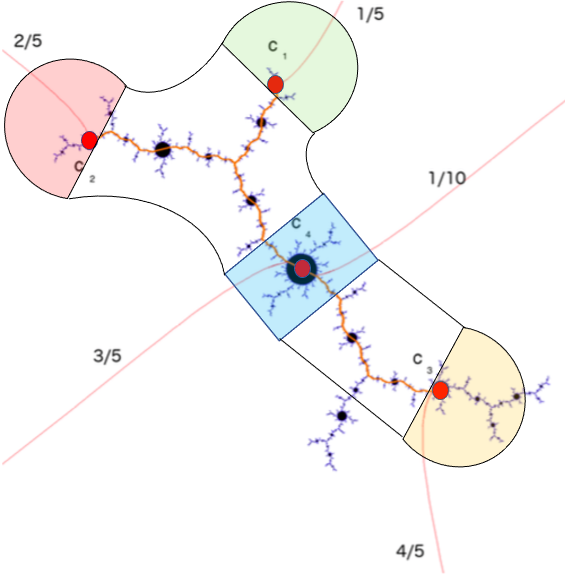}
 \caption{Thick Hubbard tree.} 
 \label{ThickH-Tree}
 \end{figure}
Figure \ref{ThickH-Tree} shows a thick Hubbard tree with four junctions (blue, green, red, yellow), one $3$-star (connecting the blue, green, and red junctions), and one simple side (connecting the blue and yellow junctions).

We define a map on the thick tree $\mathbb{T}$ as follows:
%%%%%%%%%%%%%%%%%%%%%%%%%%%%%
%	Definition
%%%%%%%%%%%%%%%%%%%%%%%%%%%%%
\begin{definition}
A \textit{thick tree map} is a continuous orientation-preserving map $F :\mathbb{T}  \rightarrow \mathbb{T} $ \ such that: 
      \begin{enumerate}   
      \item $F$ is homeomorphism onto its image;
       \item if $J$ is a junction of $\mathbb{T}$ then $F(J)$ is contained in a junction;
%       \item in each connected component of $s_i \cap F^{-1}(s_j)$, where $s_i$ and $s_j$ are strips, $F$ contracts vertical coordinates uniformly by a factor $\mu_j < 1$ and expands horizontal coordinates uniformly by a factor $\lambda_j >1$,
       \item if $J$ and $J'$ are junctions such that $F(J) \subseteq J'$ then $F(\partial J \backslash \partial \mathbb{T}) \subseteq \partial J' \backslash \partial \mathbb{T}$;
%       \item  if $J$ is a junction with $F^n(J) \subseteq J$ for some $n \geq 1$ then $J$ has an attracting periodic point of period $n$ in its interior whose basin contains the interior of $J$.
\item if $X$ is an $n$-star connecting junctions $J_1, \dots, J_k$, then $F(X)$   is an $n$-star connecting $F(J_1), \dots, F(J_k)$.
       \end{enumerate}
   \end{definition}
       
 Moreover, Figure \ref{Thick-map} shows a thick tree map of a Hubbard tree that maps the critical junction in blue to the green junction; it maps the green junction to the red junction and the red junction to the yellow junction; and it pulls the yellow junction to the blue one.     
\begin{figure}[h]
\centering
\includegraphics[width=.3\textwidth]{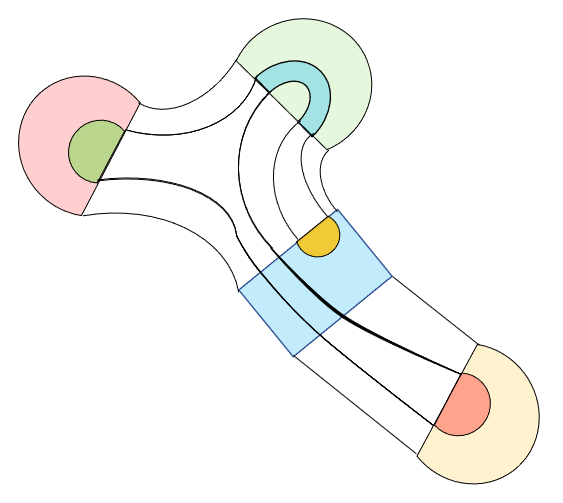}
 \caption{Thick tree map.}
 \label{Thick-map}
    \end{figure}
%%%%%%%%%%%%%%%%%%%%%%%%%%%%%
%	Tree-Like Train tracks
%%%%%%%%%%%%%%%%%%%%%%%%%%%%%
\subsection{Tree-Like Train Tracks}
\begin{definition}
We define a \textit{tree-like train track} to be a family of vertices and curves embedded on a surface such that the following hold:
\begin{enumerate}
\item there are finitely many vertices of two types, called \textit{switches} and \textit{branch points};
\item there are countably many curves of two types called \textit{edges} and \textit{loops};
\item away from vertices, the curves are smooth and do not intersect;
\item the endpoints of each curve are vertices; 
\item if the endpoints of a curve are the same point, the curve is called a \textit{loop}. Otherwise, it is called an \textit{edge};
\item at branch points only edges can meet and at switches edges and loops can meet;
\item at each switch, countably many curves meet with a unique tangent line, with one edge entering from one direction and the remaining curves entering from the other direction;
\item the union of all edges is a topological tree.
\end{enumerate}	
\end{definition}
 \begin{figure}[h]
\centering
\includegraphics[width=.3\textwidth]{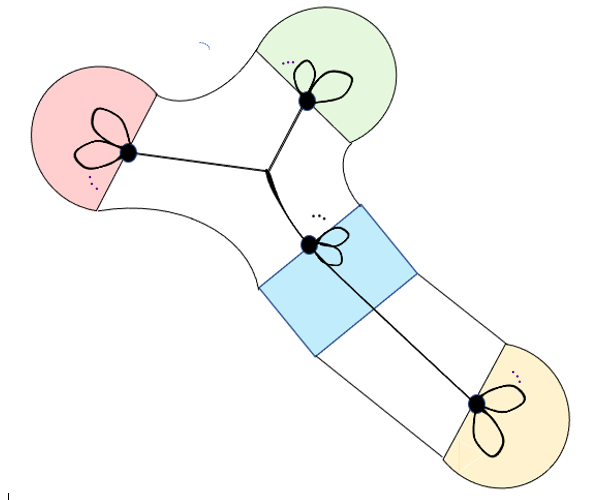}
\caption{Tree-like train track.}
\label{train-track}
\end{figure}  

%%%%%%%%%%%%%%%%%%%%%%%%%%%%%
%	Main Theorem
%%%%%%%%%%%%%%%%%%%%%%%%%%%%%
\section{Main Theorem} \label{Main}
Consider a quadratic polynomial $f:\mathbb{C} \rightarrow \mathbb{C}$ with critical point $c_0$ of period $p$.
Let $T$ be a Hubbard tree of $f$ such that it is non-degenerate, i.e. $c_0$ is not an endpoint of $T$. The critical point divides the tree into two connected components which we call \textit{half-trees}. Denote the half-tree that contains $f(c_0)$ as the \textit{upper branch} $B_u$ and the other one as the \textit{lower branch} $B_l$.
\begin{definition}
The tree $T$ is said to be \textit{crossing free} if the embedding of $T$ in the plane $\mathbb{C}$ is such that the image $f(B_l)$ of the lower branch can be isotoped, while fixing the endpoints of $T$, into a tree that does not intersect the interior $\overset{\circ}f(B_u)$ of the image of the upper branch.
\end{definition}
  \begin{figure}[h]
     \centering
     \begin{subfigure}[b]{0.45\textwidth}
         \centering
         \includegraphics[width=\textwidth]{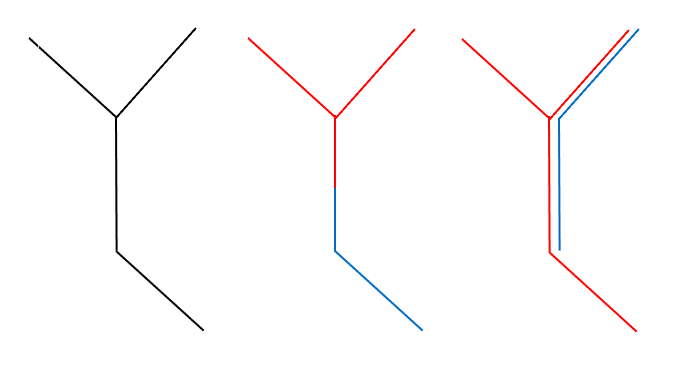}
         \caption*{Crossing-free Hubbard tree.}
     \end{subfigure}
     \hfill
     \begin{subfigure}[b]{0.45\textwidth}
         \centering
         \includegraphics[width=\textwidth]{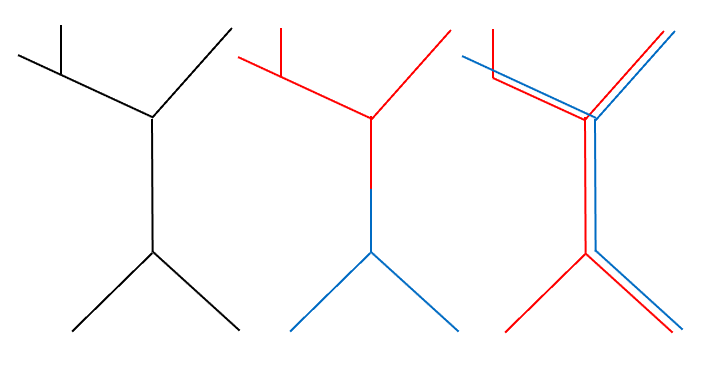}
         \caption*{Hubbard tree with crossing.}
     \end{subfigure}
     \caption{}
        \label{Examples}
    \end{figure}

We will start with some lemmas about crossing-free Hubbard trees.
%%%%%%%%%%%%%%%%%%%%%%%%%%
%     Lemma 1
%%%%%%%%%%%%%%%%%%%%%%%%%%
\begin{lemma}\label{L1}
Let $T$ be a crossing-free non-degenerate Hubbard tree with a branch point $P$. Then the branch point $P$ cannot be mapped to a fixed branch point.
\end{lemma}
\begin{proof}
	Assume by contrary that there is a fixed branch point $P_{0}$ such that $f(P) := P_0.$
	Note that $P$ and $P_0$ cannot be in the same half-tree since the restriction of $f$ to each of the half-trees is injective.
	Looking at the images of the neighbourhoods of $P$ and $P_0$, we have the following:
	\begin{itemize}
	\item 	the images of neighbourhoods of $P$ and $P_0$ are neighbourhoods of $P_0$
	\item any sufficiently small neighbourhood of $P_0$ is homeomorphic to a star $S$ with $n>2$ branches 
	\item there is no homotopy that takes the star away from itself while staying in an $\epsilon$ neighbourhood of the tree.	
 	\end{itemize}
This contradicts the crossing-free assumption.
\end{proof}
\begin{lemma} \label{L2}
In a Hubbard tree, every branch point is pre-periodic. 
Moreover, if the tree is crossing free, then every pre-fixed branch point is fixed.
\end{lemma}
\begin{proof}
	Note that in any Hubbard tree, every branch point is mapped to a branch point. Since there are finitely many branch points, every branch point is pre-periodic. 
The second claim follows from Lemma \ref{L1}.
\end{proof}
\begin{conjecture}
	If the Hubbard tree is crossing-free, then there is only one branch point.
\end{conjecture}
Now we will give the main definition of extension of a tree map to a generalized pseudo-Anosov homeomorphism.
\begin{definition} \label{extension}
 Let $\varphi :S\rightarrow S$ be a generalized pseudo-Anosov homeomorphism of a surface $S$, let $\mathcal{F}$ be one of its invariant foliations, and let $f:T\to T$ be a Hubbard tree of a post-critically finite quadratic polynomial.
Let $T_0$ be the tree $T$ minus the endpoints.
 Suppose the following holds:
 \begin{enumerate}
 \item 	 there exists an open, connected, dense subset $S_0$ of $S$;
 \item the quotient map, given by collapsing each leaf of the restriction of $\mathcal{F}$ to $S_0$, yields a surjective, continuous map $\pi:S_0\to T_0$;
\item the following diagram commutes:

 \[
 \begin{tikzcd} S_1 \arrow{r}{\varphi} \arrow[swap]{d}{\pi} & S_0 \arrow{d}{\pi} \\  T_1 \arrow{r}{f} & T_0  
 \end{tikzcd} 
 \]
 where $S_1 := \varphi^{-1}(S_0)$ and $T_1 := f^{-1}(T_0)$. 
 \end{enumerate}
Then we say that $\varphi$ is an \textit{extension} of $f$. We say that a polynomial map $f$ is \textit{extendable} if there exists an extension $\varphi$ as above.
 \end{definition}

Now we will prove the first part of our main theorem.
%%%%%%%%%%%%%%%%%%%%%%%%%%%%
%       Main Theorem
%%%%%%%%%%%%%%%%%%%%%%%
\begin{theorem}\label{main-theorem}
Let $f$ be a post-critically finite, superattracting quadratic polynomial and let $T$ be its Hubbard tree. If $f:T\to T$ is crossing-free, then it is extendable to a generalized pseudo-Anosov homeomorphism $\varphi$.
 Moreover, if $\lambda$ is the dilatation of $\varphi$, then $\log \lambda$ equals the core entropy of $f$.
\end{theorem}
\begin{proof}
This is a proof by construction.
Let $T$ be a crossing-free non-degenerate Hubbard tree. Let $c_0$ be the critical point with period $p$. Since $T$ is a non-degenerate Hubbard tree, then $c_0$ is located in the middle of $T$ between the upper branch $B_u$ and the lower branch $B_l$. 

\begin{figure}[h!]
\centering
\includegraphics[width=.28\textwidth]{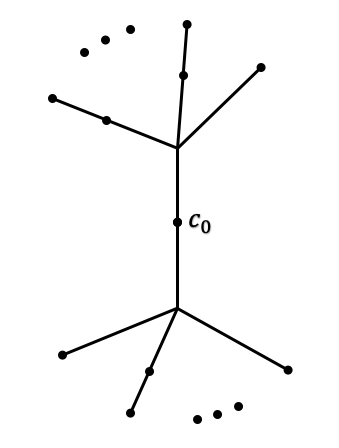}
\caption{ }
\end{figure} 

Since $T$ is a crossing-free Hubbard tree, then there exist:
 $$\Phi_u:f(B_u)\times [0,1] \rightarrow N_\epsilon(f(B_u)) \text { defined by }$$  $$\Phi_u(p,t)=p, \ \ \forall p\in \partial f(B_u), \ \ \forall t,$$ and 
$$\Phi_l:f(B_l)\times [0,1] \rightarrow N_\epsilon(f(B_l)) \text { defined by }$$ $$\Phi_l(p,t)=p, \ \ \forall p\in \partial f(B_l), \ \ \forall t,$$ such that $$\Phi_u(f(B_u)\backslash \partial f(B_u),1)\cap \Phi_u(f(B_l)\backslash \partial f(B_l),1)=\emptyset.$$
Now the construction will take the following steps:

\textbf{Step 1.} We thicken the tree $T$ into a thick tree $\mathbb{T}$ that consists of junctions $J$'s and sides $S$'s. Junctions are topologically discs that represent the critical orbit. So there are $p$ junctions in $\mathbb{T}$. In between junctions, there are sides $S$'s which are topologically rectangles such that the following holds:
\begin{itemize}
\item If the connected component $cc(\mathbb{T}\backslash \{J's,T\})$ 	contains a single edge of $T$ with no branch points, then $S=cc(\mathbb{T}\backslash \{J's,T\})$.
\item If $cc(\mathbb{T}\backslash \{J's,T\})$ contains a branch point in $T_\theta$ of valence $n$, then we connect the branch point to each of the boundary component with lines $l_i$, for $1\leq i \leq n$. Moreover, $cc(\mathbb{T}\backslash \{J's,T\})$  will contain $n$ sides $S_i, \ 1\leq i\leq n$ such that $S_i$ has $l_i \cup l_{i+1}$ as one side and a junction as the opposite side.

Around an $n$-branch point, $n$ sides form an $n$-star  
\end{itemize}
Note that in $\mathbb{T},$ there are $p$ junctions and $p+b-1$ sides, where $b$ is the number of branch points.

Now define $\tilde{f}: f(T) \rightarrow N_\epsilon(f(T))$ by 
\[ \tilde{f}(x) :=\begin{cases} 
       \Phi_u(x,1) &\text{if } x\in f(B_u) \\
       \Phi_l(x,1) & \text{if } x\in f(B_l) \\
     
   \end{cases}.
\]

 Starting with the thick tree $\mathbb{T}$, we embed the images $\tilde{f}(B_u)$ and $\tilde{f}(B_l)$ such that $J_{Red} :=\tilde{f}(B_u)\cap J_{f(c_0)}$ and $J_{Blue} :=\tilde{f}(B_l)\cap J_{f(c_0)}$ intersect at $f(c_0)$. We smoothen this point such that $J_{Red}\cup J_{Blue}$ is diffeomorphic to half a circle. 
 
\textbf{Step 2.} We define a thick map $F:\mathbb{T} \rightarrow \mathbb{T}$ such that it satisfies the following:
 \begin{enumerate}
 \item $F(\mathbb{T})\subseteq \mathbb{T}.$
 
 \item The junction $J_{c_0}$ around the critical point $c_0$ is mapped homeomorphically into the junction $J_{c_1}$ around the critical value $c_1$ such that $F(J_{c_0})\subseteq J_{c_1}$ and the two sides $J_{c_0,u}$ and $J_{c_0,l}$ of $\partial J_{c_0} \cap \overset{\circ}{\mathbb{T}}$ are mapped to disjoint segments of the one side of  $\partial J_{c_1} \cap \overset{\circ}{\mathbb{T}}$ in the following way:
 	\begin{itemize}
     \item  $F(J_{c_0,u})$ intersects $J_{Red}$ and $F(J_{c_0,l})$ intersects $J_{Blue}$,
      \item  $F:J_{c_0} \rightarrow F(J_{c_0})\subseteq J_{c_1}$ is an embedding,
 		\item  and the union $J_{Red} \cup J_{Blue}$ is contained in $F(J_{c_0})$.
 	\end{itemize}
 \item Junctions that do not contain the critical point $c_0$ are mapped  homeomorphically into junctions. More precisely, if $J_i$ contains $c_i$ with $i\in \{1,\dots, p-1\}$, then $F(J_i)$ contains $f(c_i)$ and  $F(J_i)\subseteq J_j$ for $j\in \{0,1,\dots, p-1\}$. There are three cases for such junctions:
   \begin{itemize}
     \item[] \textbf{Case 1.} If end junction, end point of $T$, is mapped to end junction, then the boundary $\partial$ is mapped a boundary $\partial$, and the rest of the junction is embedded in the junction.
     \item[] \textbf{Case 2.} If end junction is mapped to an inside junction, then the boundary $\partial $ is mapped the boundary part where it meets $\tilde{f}(B_\alpha).$
      \item[] \textbf{Case 3.} If an inside junction is mapped to an inside one, then the boundary $\partial$ sides will be mapped to the corresponding boundary sides such that they preserve the orientation of the original map.
    \end{itemize}
\item Sides between junctions will be mapped as follows:
\begin{enumerate}
  \item[] \textbf{Case 1.} If a side $S_i$ is connecting two junctions $J_m$ and $J_n$, then it is mapped to a side $F(S_i)$ that connects $F(J_m)$ and $F(J_n)$ such that:
    \begin{itemize}
      \item the inner boundary $\partial S_i \cap \overset{\circ}{\mathbb{T}}$ is mapped to $\partial F(J_m) \cap \overset{\circ}{\mathbb{T}}$ and $\partial F(J_n) \cap \overset{\circ}{\mathbb{T}}$,
      \item the outer boundary $\partial S_i \backslash (\partial S_i \cap \overset{\circ}{\mathbb{T})}$ is mapped to disjoint arcs that connect the boundary parts of the images of the junctions accordingly,
       \item and $F(S_i)$ is a rectangle that is a neighbourhood of $\tilde{f}(f(S_i\cap T)).$
    \end{itemize}
  \item[] \textbf{Case 2.} Around an $n$-branch point, $n$ sides form an $n$-star 
  
  . An $n$-star is mapped to an $n$-star, specifically, $n$ sides are mapped to $n$ rectangles each of which has $F(\partial S_i \cap \partial J_i)$ as one side.
 \end{enumerate}
\end{enumerate}
 
 Let $T$ be a tree connecting the critical orbit and representing the Hubbard tree $T_\theta$. 
 For each $S_i$ there is a homeomorphism $h_i: S_i \rightarrow [0,1]\times [0,1]$ with a rectangle so that $h_i(T\cap S_i)=\{\frac{1}{2}\}\times [0,1]$. That is the tree is sent to the middle horizontal line.
 Let $P:[0,1]\times [0,1]\rightarrow  \{1/2\}\times [0,1]$ be defined as
$$P(x,y) :=(\frac{1}{2},y).$$

Define $\pi_i: S_i \rightarrow T\cap S_i$ on each $S_i$ to be given by
$$\pi_i := h_i^{-1} \circ P \circ h_i$$

Define a continuous map $\xi_j: J_j \rightarrow J_j$ on each junction to be $$\xi_j :=\pi_i \text{ on } \partial J_j \cap \partial S_i$$ for every side $S_i$ that intersects the junction $J_j$, and so that the restriction of $\xi_j$ to the interior $J_j \setminus \cup (\partial S_i)$ is a homeomorphism onto its image.
 
 We define an isotopy map on $\mathbb{T}$ as follows:
 $$\psi :\mathbb{T} \rightarrow  \mathbb{T}$$ 
$$ \psi(z) := 
	\left \{ 
		\begin{array}{lll}
		\pi_i(z)& \mbox{if}\;\; z \text{ is in the side } S_i\\
		\xi_j(z)	 & \mbox{if}\;\; z \text{ is in the junction } J_j. 
		\end{array}
	\right.
$$ 
 
To keep track of the folding, we construct a tree-like train track as follows:

\begin{itemize}
\item First we let $\tau_0 :=(B_u \cup B_l)\setminus \overset{\circ}{J_0}$ be the disconnected tree-like train track with switches that are located at the intersection of the tree $T$ and the boundary of the junction $\partial J_0$. 
\item Then, we apply the map $F$ to $\tau_0$.

\item Apply $\psi $ to $F(\tau_0)$. 
\item Denote the resulting by $\tau_1$.

\item Now we continue by applying $F$ to $\tau_1$ and repeat the above process.

\item This way, we get;
$$\tau_0 \subset \tau_1 \subset \dots  \subset \cup \tau_i.$$
\end{itemize}

More precisely, let $\Phi :\mathbb{T} \to \mathbb{T} $ be defined as 
$$\Phi := \psi \circ F.$$
Define $\tau_n :=\Phi ^n(\tau_0)$. Since $\tau_0 \subset \tau_1,$ then $ \tau_{n} \subset \tau_{n+1} \ \forall n.$
Let $$\tau_\infty :=\bigcup_{n=0}^\infty \tau_n.$$
Then by construction 
\begin{align*}
	\Phi (\tau_\infty)&=\Phi (\bigcup_{n=0}^\infty \tau_n)
	=\bigcup_{n=0}^\infty \Phi (\tau_n)
	=\bigcup_{n=1}^\infty \tau_n
	=\tau_\infty.
\end{align*}
\begin{figure}[h]
\centering
 \includegraphics[scale=.38]{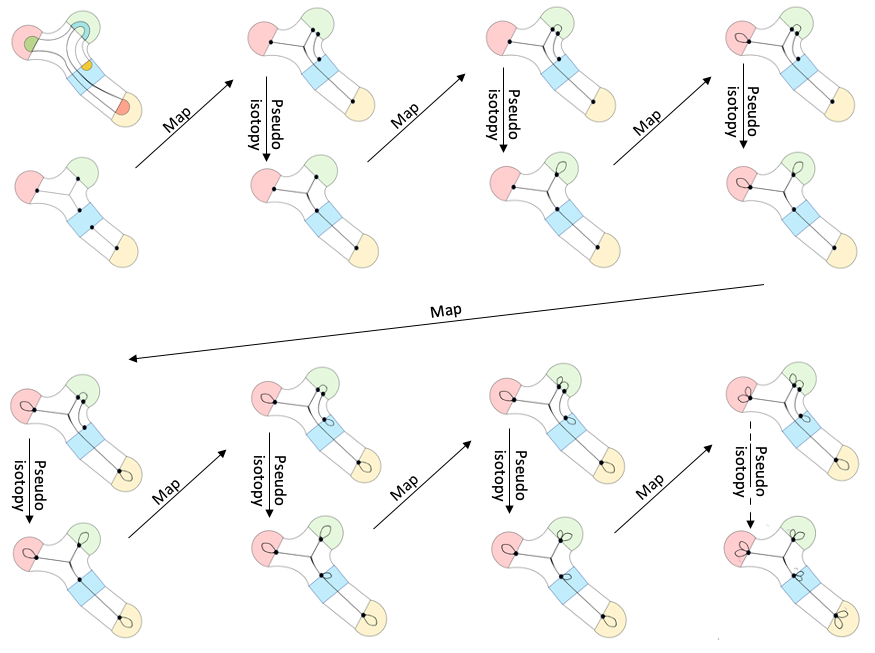}
 \caption{How to produce a tree-like train track using pseudo-isotopies.}
  \label{tree-like}
 \end{figure} 
 
Let $t_i := \pi(S_i), \ 1\leq i\leq s$ be the edges of the tree-like train track $\tau_\infty.$ Define a matrix $M=(m_{ij})$ where $m_{ij}$ equals the number of times $t_i$ crosses $\Phi(t_j)$ for $1\leq i,j \leq s$.

Let $\lambda$ be the leading eigenvalue and $x=(x_1,\dots , x_s)$ and $y=(y_1,\dots , y_s)$ be the right- and left-eigenvalues corresponding to $\lambda$, respectively.
Then we construct rectangles $R_i, 1 \leq i\leq s,$ of dimensions $x_i \times y_i$.
Then we glue these rectangles. To show that this gluing is possible, we will look at the various cases of the branch points. First we will consider the case where the valence is an odd number $n$. Then for the case of an even valence, we will consider the three options of the branch point, the case of a fixed branch point, the case of periodic branch point with period $k>1$, and the case of pre-periodic branch point with pre-period $l>1$.

The next lemma shows that it's possible to glue the rectangles around a branch point with odd valence.
%%%%%%%%%%%%%%%%%%%%%%%%%%%
%     Lemma 3
%%%%%%%%%%%%%%%%%%%%%%%%%%%%
\begin{lemma}
Let $\tau$ be a tree-like train track with a branch point $p$ of valence $n$, where $n$ is odd. Let $R_i, i=1,\dots ,n,$ be $n$ rectangles constructed for the edges around the branch point $p$ in either a clockwise or counter-clockwise direction. Let $x_i$ and $y_i$ be the sides of $R_i$. Then we can glue all of $Y_i, i=1,\dots ,n,$ around the branch point such that each $y_i$ is divided into two subsegments $z_i$ and $w_i$, and $w_i$ is glued to $z_{i+1}$.
\end{lemma}
\begin{proof}
We need to show that the gluing of the subsegments is possible. 
Let the rectangles $R_i, i=1,\dots ,n,$ be organized in a counter-clockwise order such that the sides $y_1, \dots, y_n$ form an $n$-polygon.
So we want to find if the following system has a well-defined solution:
  \begin{equation} \label{sys1}
	\begin{array}{ll}
z_1+z_2 &= y_1\\
z_2+z_3 & = y_2\\
	& \vdots \\
z_n+z_1 & = y_n
\end{array} 	
  \end{equation}
Since $n$ is odd: The above system of equations (\ref{sys1}) has the augmented matrix $$\begin{bmatrix}
	1& 1& 0 &\dots &&0\\
	0& 1& 1 &\dots& &0\\
	\vdots & & \ddots && &\vdots\\
	0& 0& \dots &1&1 &0\\
	0& 0& \dots &&1 &1\\
	1& 0& \dots &&0 &1
\end{bmatrix},$$ such that the following hold

\begin{align*}
	\begin{bmatrix}
	1&-1&\dots &1&-1&1
\end{bmatrix}
&\begin{bmatrix}
	1& 1& 0 &\dots &&0\\
	0& 1& 1 &\dots& &0\\
	\vdots & & \ddots && &\vdots\\
	0& 0& \dots &1&1 &0\\
	0& 0& \dots &&1 &1\\
	1& 0& \dots &&0 &1
\end{bmatrix}\\ &=\begin{bmatrix}
2z_1&0&\dots&0	
\end{bmatrix}.
\end{align*}

Hence we have $$y_1-y_2+\dots +y_{n-1}-y_n=2z_1.$$
 Since $z_1$ must be positive, we need the following condition to hold
 $$y_1+\dots +y_{n-1}> y_2 + \dots+y_n.$$
 But this is true since $y_i>y_{i+1}, i=1,\dots, n-1,$ given that $y_i=\lambda y_{i+1}$. This completes the proof.
\end{proof}

In the case of gluing the rectangles around a branch point with an even valence, the next lemma shows that it is possible for a fixed branch point.
%%%%%%%%%%%%%%%%%%%%%%%%%%%%
%     Lemma 4
%%%%%%%%%%%%%%%%%%%%%%%%%%%%
\begin{lemma}
Let $\tau$ be a tree-like train track with a fixed branch point $p$ of valence $n$. Let $R_i, i=1,\dots ,n,$ be $n$ rectangles constructed for the edges around the branch point $p$ in either a clockwise or counter-clockwise direction. Let $x_i$ and $y_i$ be the sides of $R_i$. Then we can glue all of $Y_i, i=1,\dots ,n,$ around the branch point such that each $y_i$ is divided into two subsegments $z_i$ and $w_i$, and $w_i$ is glued to $z_{i+1}$.
\end{lemma}
\begin{proof}
Again here we need to show that the gluing of the subsegments is possible. 
Let the rectangles $R_i, i=1,\dots ,n,$ be organized in a counter-clockwise order such that the sides $y_1, \dots, y_n$ form an $n$-polygon.
So we want to find if the following system has a well-defined solution:
\begin{equation} \label{sys2}
	\begin{array}{ll}
z_1+z_2 &= y_1\\
z_2+z_3 &= y_2\\
	& \vdots \\
z_n+z_1 &= y_n
\end{array}
\end{equation}
Since $n$ is even, then the above system of equations (\ref{sys2}) has the augmented matrix $$\begin{bmatrix}
	1& 1& 0 &\dots &0\\
	0& 1& 1 &\dots &0\\
	\vdots & & \ddots & &\vdots\\
	0& 0& \dots &1 &1\\
	1& 0& \dots &0 &1
\end{bmatrix},$$ with rank $n-1$. Hence, the $\dim(Ker)=1$ with $(1,-1,\dots ,1,-1)$ in the Kernel. That is 

\begin{equation}\label{Yi-Even}
	y_1-y_2+\dots +y_{n-1}-y_n=0.
\end{equation}
	 Also, the Markov matrix of the system associated to the tree-like train track is given by the matrix 
	$$ M=\begin{bmatrix}
	 \begin{matrix}
	 	0& 1& 0 &\dots &0\\
	0& 0& 1 &\dots &0\\
	\vdots & & \ddots & \ddots&\vdots\\
	0& 0& \dots &0 &1\\
	1& 0& \dots &0 &0
	 \end{matrix}
	 & \vline & B\\
    \hline A
    & \vline & 
  \begin{matrix}
	 	& & &\\
		& \ C & &\\
		& & &		
	 \end{matrix}
	\end{bmatrix},$$ 
where the upper left block is $n\times n$ coming from the half-tree that contains the branch point, say, the upper branch $B_u$. Then the block $C$ is coming from the other half-tree, in this case the lower branch $B_l$, say it's $m\times m$.

We claim that
\begin{equation}\label{3}
	Mv=-v
\end{equation}
 where $v=(1,-1,\dots, 1,-1,0,\dots, 0)$ is an eigenvector associated to the eigenvalue $-1$.
 
 The claim is true because block $A$ has all zeros except for the first row which is given exactly by $\begin{bmatrix}
	1&1&0& \dots &0 \end{bmatrix}.$ In fact, the rows of block $A$ represent all the edges that are disjoint from the branch point. 
Block $A$ contains ones only on places representing sides containing a pre-image $q\neq p$ of the branch point. Since the tree-like train track has a fixed branch point then by Lemma \ref{L1}, there is no other branch point that's mapped to the fixed point. Hence the other pre-image $q$ is not a branch point. So $q$ lies on some edge $e_j, j> n$ in the lower branch. The edge $e_j$ maps to the union of two edges containing the critical point, and since the tree is crossing free, the two edges are consecutive. This gives rise to consecutive ones in row $j$ of block $A$. This proves the claim. 
\newline

On the other hand, the transpose is given by the matrix 
$$M^\intercal=\begin{bmatrix}
	 \begin{matrix}
	 	0& 0& 0 &\dots &1\\
	 	1& 0& 0 &\dots &0\\
	\vdots &  \ddots & \ddots& &\vdots\\
	0& 0& \dots &0 &0\\
	0& 0& \dots &1 &0
	 \end{matrix}
	 & \vline & A\\ 
	 \hline B
  & \vline &
  \begin{matrix}
  	& & &\\
  	& \ C& &\\
  	& & &		
	 \end{matrix}
	\end{bmatrix}$$

Let $\lambda$ be the leading eigenvalue for $M^\intercal$ and since $(y_1, \dots, y_{n+m})$ is the associated eigenvector, then we have
\begin{equation} \label{4}
M^\intercal y=\lambda y
\end{equation}

i.e.
$$\begin{bmatrix}
	 \begin{matrix}
	 	0& 0& 0 &\dots &1\\
	1& 0& 0 &\dots &0\\
	\vdots &  \ddots & \ddots& &\vdots\\
	0& 0& \dots &0 &0\\
	0& 0& \dots &1 &0
	 \end{matrix}
	 & \vline & A \\
\hline B
  & \vline &
  \begin{matrix}
	 	& & &\\
		& 
		 C& &\\
		& & &
			
	 \end{matrix}
	\end{bmatrix}
 \begin{bmatrix}
 y_1\\ y_2\\ \vdots \\ y_n\\ \vdots\\ \vdots \\ y_{n+m}
 \end{bmatrix}=\lambda  \begin{bmatrix}
y_1\\ _2\\ \vdots\\ y_n\\ \vdots\\ \vdots \\ y_{n+m}
 \end{bmatrix}$$ 

(\ref{3}) and (\ref{4}) give the following
\begin{align*}
v^\intercal	(M^\intercal y) &= v^\intercal( \lambda y)\\
(Mv)^\intercal y &= \lambda (v^\intercal y)\\
-\langle v,y \rangle &= \lambda \langle v,y \rangle ,
\end{align*} which implies 
\begin{equation*}
(\lambda+1) \langle v,y \rangle =0.
\end{equation*}
Since $\lambda \neq -1,$ we get 
\begin{equation}
	\langle v,y \rangle=0.
\end{equation}

This gives what we wanted to show as in (\ref{Yi-Even}).	 
\end{proof}

The following lemma shows that the gluing is possible in the case of a periodic branch point with period $k>1$.
%%%%%%%%%%%%%%%%%%%%%%%%%%%%
%     Lemma 5
%%%%%%%%%%%%%%%%%%%%%%%%%%%%
\begin{lemma}
Let $\tau$ be a tree-like train track with a periodic branch point $p$ of valence $n$ and period $k>1$. Let $R_i, i=1,\dots ,n,$ be $n$ rectangles constructed for the edges around the branch point $p$ in either a clockwise or counter-clockwise direction. Let $x_i$ and $y_i$ be the sides of $R_i$. Then we can glue all of $Y_i, i=1,\dots ,n,$ around the branch point such that each $y_i$ is divided into two subsegments $z_i$ and $w_i$, and $w_i$ is glued to $z_{i+1}$.
\end{lemma}
\begin{proof}
Since $p$ is a periodic branch point of valence $n$ and period $k$, then this part of the Hubbard tree will consist of $k$ branches as shown in Figure \ref{periodic-branch-pt}. The edges around $p$ are denoted by $e_1, \dots, e_n$ in the counter-clockwise direction where $e_1$ is the edge connecting the branch point $p$ to other part of the Hubbard tree. Call the branch point $p$ with the edges $e_1,\dots, e_n$ a $p$-star denoted by $S_p$. 
 \begin{figure}[h]
\centering
\includegraphics[width=0.5\textwidth]{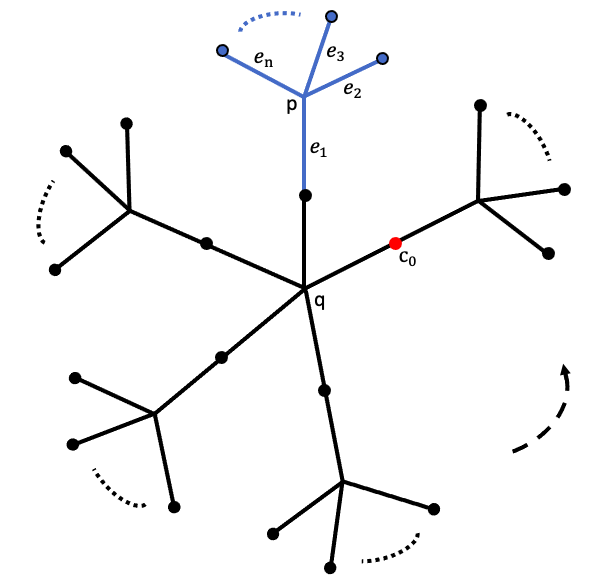}
 \caption{A periodic branch point $p$ of valence $n$ and period $k$.}
 \label{periodic-branch-pt}
 \end{figure}   
We look at the dynamics of the $k$-th iteration of the map $f$ on all edges $e_i$ of the Hubbard tree $T$ and choosing the ones intersecting the $p$-star. This gives the following:
$$\{f^k ( e_1),f^k ( e_2),\dots ,f^k ( e_n)\}=\{e_1,e_2, \dots,e_n\}$$ 
If there is an edge $e_j$ in the Hubbard tree $T$ such that $f^k ( e_j)$ intersects the $p$-star, then the intersection $f^k ( e_j)\cap S_p = e_i\cup e_{i+1}$ for some $i=1,\dots, n-1$. This is because of the condition that the Hubbard tree $T$ is crossing free.
This will show in the Markov matrix for $f^k$ as an upper block which is an $n\times n$ matrix as in the proof of Lemma 4 and all rows in block $B$ are zeroes except for some with consecutive ones. 
$$\begin{bmatrix}
	 \begin{matrix}
	 	0& 0& 0 &\dots &1\\
	1& 0& 0 &\dots &0\\
	\vdots &  \ddots & \ddots& &\vdots\\
	0& 0& \dots &0 &0\\
	0& 0& \dots &1 &0
	 \end{matrix}
	 & \vline & A\ \ \ \\
\hline B
  & \vline &
  \begin{matrix}
	 	& & &\\
		& 
		 C& &\\
		& & &		
	 \end{matrix}
	\end{bmatrix}$$ 
Following the same argument in Lemma 4, we get the result.
This completes the proof.

\end{proof}

%%%%%%%%%%%%%%%%%%%%%%%%%%%%
%     Lemma 6
%%%%%%%%%%%%%%%%%%%%%%%%%%%%
\begin{lemma}
Let $\tau$ be a tree-like train track with a pre-periodic branch point $p$ of valence $n$ and period $k>1$. Let $R_i, i=1,\dots ,n,$ be $n$ rectangles constructed for the edges around the branch point $p$ in either a clockwise or counter-clockwise direction. Let $x_i$ and $y_i$ be the sides of $R_i$. Then we can glue all of $Y_i, i=1,\dots ,n,$ around the branch point such that each $y_i$ is divided into two subsegments $z_i$ and $w_i$, and $w_i$ is glued to $z_{i+1}$.
\end{lemma}
\begin{proof}
Since $p$ is pre-periodic, then there is a number $l$ such that $f^l(p)$ is periodic and by the last Lemma, hence the gluing is possible for $f^l(p)$. Then we pull back and that will give the result.
\end{proof}

The last three lemmas showed that the gluing of the $n$-rectangles around the branch point $p$ is possible. 

So now we organize all rectangles following the tree-like train track such that the edges of the tree-like train track represent the $x$-direction of the rectangles.
Then we glue each two adjacent rectangles in the $y$-direction. For those $n$-rectangles around the branch point $p$, the gluing is done in the way described by the previous Lemmas such that a hyperbolic angle is formed. If there is no branch point, then we glue the two rectangles such that they align in the $x$-direction on the side opposite of the loops as in Figure \ref{Aligning-rectangles}.
 \begin{figure}[ht]
\centering
\includegraphics[width=0.3\textwidth]{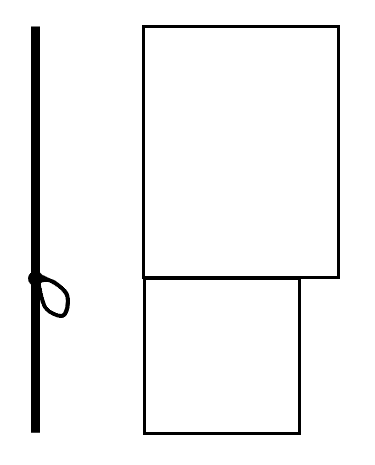}
 \caption{Aligning rectangles according to loops direction.}
 \label{Aligning-rectangles}
 \end{figure} 

Let $\mathcal{R}$ be the union of all the rectangles $R_i$.
Define $H:\mathcal{R}\to \mathcal{R}$ by 
$$H(R_i) :=\tilde{R}_j$$ such that the following hold:
\begin{itemize}
\item 	$\tilde{R}_j$ has dimension $x_i \lambda \times y_i \lambda^{-1}$,
\item $\tilde{R}_j$ is located in the part of $\mathcal{R}$ that corresponds to the image $F(e_i)$ in $F(\tau_\infty )$, say $\tilde{R}_j\subseteq \cup  R_k$ for some $k$, and
\item  $\tilde{R}_j$ align in the $x$-direction with $\cup  R_k$ on the side following the image $F(\tau_\infty )$ of the tree-like train track.  
\end{itemize}

 \begin{figure}[ht]
\centering
\begin{tikzpicture}
 \node(a)
{\includegraphics[width=3cm]{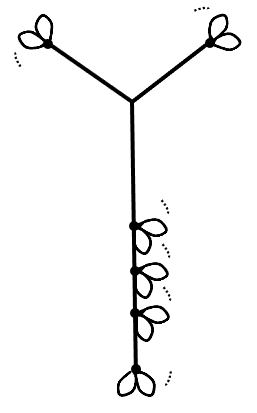}};
 \node[right=of a](b)
{\includegraphics[width=4cm]{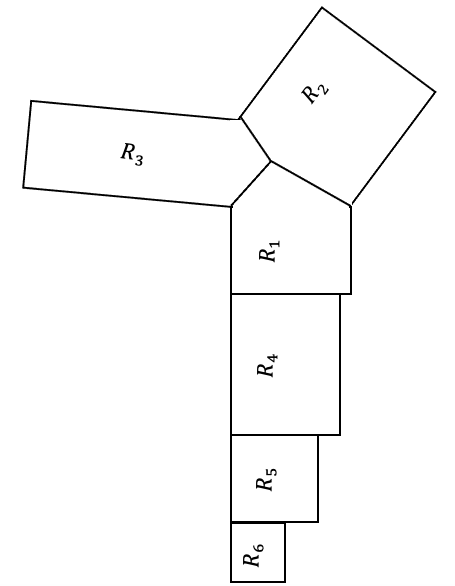}};
\end{tikzpicture}
\caption{Gluing all the rectangles using the tree-like train track to get the surface $\mathcal{R}$.}
\label{Train-Surface}
\end{figure}

\begin{figure}[ht]
\centering
\begin{tikzpicture}[>=latex,node distance=2em]
 \node(a)
{\includegraphics[width=4cm]{Ex1-10.png}};
 \node[right=of a](b)
{\includegraphics[width=4cm]{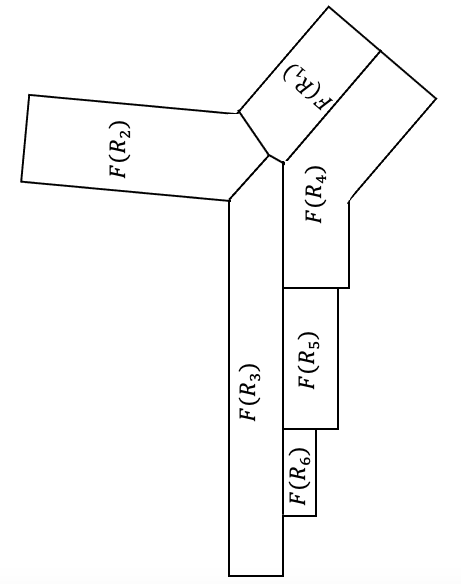}};
\draw[->] (a) -- (b);
\end{tikzpicture}
\caption{Applying the induced map $H$ on $\mathcal{R}$.}
\label{R-H(R)}
\end{figure}

$H$ acts homeomorphically on $\mathcal{R} \backslash (\partial \mathcal{R}\cup L_0)$ where $L_0$ is the side representing the critical point.
 We then glue the $x$-direction of the boundary $\partial \mathcal{R}$ by the following steps:
 \begin{enumerate}
 \item 	first identify the pre-images of each segment in the $x$-direction of the image of the boundary $H(\partial \mathcal{R})$,
 \item  then glue the remaining parts of the $x$-direction of the boundary $\partial \mathcal{R}$ by following the identifications in step 1 and pull-back to the pre-images. 
 \item repeat step 2 until all the $x$-direction of the boundary $\partial\mathcal{R}$ are identified.   
 \end{enumerate}
Last we identify the rest of the $y$-direction of the boundary $\partial\mathcal{R}$, which represent the loops in the tree-like train track, as follows:

\begin{itemize}
\item 	starting from the side of the first formed loop, we pinch the side and identify the two sides of the pinch;
\item then we continue with a set of identifications by pinching the remaining part of the side where each pinch represents a loop;
\item since there are infinite loops, then there will be a set of infinite pinching or identifications;
\item to decide the exact points of the side where the identifications take place we use the following equation $$\sum_{n=0}^{\infty}{\lambda ^{-np}y_0}=y$$ where $y$ is the length of the side, $y_0$ is the length of the first identification, $\lambda$ is the leading eigenvalue, and $p$ is the period of the critical point.
\end{itemize}
We now need to show that the infinite singularities, occurring from the boundary identifications, accumulate at finitely many points. First, let us take care of the singularities formed from the loops in the tree-like train track. Since $f$ is a post-critically finite map, then there are finite number of connected segments $\tilde{S}_i$ in the boundary $\partial \mathcal{R}$ with loops. Each segment $\tilde{S}_i$ has infinite singularities. WLOG, let $\tilde{S}_i$ be a side of length $y$, and let's parametrize it as an interval $\tilde{S}_i=[0,y]$ such that $0$ corresponds to the side of $\tilde{S}_i$ that the first singularity is located. Thus the singularities will be located in the interval $[0,y_i]$ exactly at the points $$\frac{y_i}{2},\ y_i +\frac{y_i}{2\lambda^p},\ y_i +\frac{y_i}{2}+\frac{y_i}{2\lambda^{2p}},\ \dots $$
Since $$\lim_{n\to \infty} \frac{y_i}{2\lambda^{np}}=0,$$ then the singularities on $\tilde{S}_i$ accumulate at the point $$y_i +\frac{y_i}{\lambda^p}+\frac{y_i}{\lambda^{2p}}+\ \dots =y.$$
Similarly, for the singularities in the $x$-direction, there is at most one limit point on each side of the thickened tree. This shows that the resulting map $\varphi$ in the quotient is a generalized pseudo-Anosov. 
 \end{proof}

%%%%%%%%%%%%%%%%%%%%%%%%%%%%
%      Converse
%%%%%%%%%%%%%%%%%%%%%%%%%%%%
Now we will prove the converse of Theorem \ref{main-theorem}, that is the second part of the main theorem:

\begin{theorem} \label{converse}
If $\varphi$ is a generalized pseudo-Anosov map that is an extension of a non-degenerate Hubbard tree $T$ then $T$ is crossing free. 
\end{theorem}
\begin{proof}

	Since $\varphi$ is a generalized pseudo-Anosov map that is an extension of a non-degenerate Hubbard tree $T$, then the following  diagram commutes:
	\[
 \begin{tikzcd} S_1 \arrow{r}{\varphi} \arrow[swap]{d}{\pi} & S_0 \arrow{d}{\pi} \\  T_1 \arrow{r}{f} & T_0  
 \end{tikzcd} 
 \]
where $\pi:S_0\to T_0$ is the quotient map given by collapsing each leaf of one of the invariant foliations $\mathcal{F}$ of $\varphi$ to a point, and $f:T\to T$ is a continuous map.
In order to show that $T$ is crossing-free, we need to show that the embedding of $T_0$ in $S_0$ has the following property: 
the images $f(B_u)$ and $f(B_l)$ of the upper and lower branches can be isotoped, while fixing the endpoints of $T$, into two trees with disjoint interiors. 

To show this, first we define $i:T_0\to S_0$ such that $\pi \circ i := id_{T_0}$. Consider $T :=B_u \cup B_l$. Let $\tilde{B}_u :=i(B_u\cap T_0), \tilde{B}_l :=i(B_l\cap T_0)$. Since $\varphi$ is a homeomorphism on $S$, then $\varphi (\tilde{B}_u\backslash\{c_0\})\cap \varphi (\tilde{B}_l\backslash\{c_0\})=\phi.$
Let $\tilde{T} :=i(T_0)$. Let $\tilde{c_0} :=i(c_0)$ and $\tilde{c_1} :=\varphi (\tilde{c_0})$. Let $d_{\mathcal{F}}(x,y)$ be the distance between the leaves of $\mathcal{F}$ containing $x$ and $y$.

Since $S_0$ can be embedded in the plane, we can use coordinates of the plane to define a homotopy as follows:
$$H:\tilde{T}\times [0,1] \to S_0\ \text{ given by}$$
$$H(x,0) :=\tilde{f}(x)=i(f(\pi(x)))\ \text{ and }\ H(x,1) := d_{\mathcal{F}}(x,\tilde{c_1})\varphi (x)\ \text{ for all }\ x\in \tilde{T},$$ and $$H(x,t)=(1-t)H(x,0)+t \ H(x,1), \ t\in [0,1].$$
Since $S_0$ can be embedded in the plane $\mathbb{C}$, then $H$ is  a homotopy in $\mathbb{C}$.
We can extend the homotopy to the endpoints of $T$ and note that it induces an isotopy on the upper and lower branches because by definition, at each time $t$ every leaf will contain exactly one point of $H(\tilde{B}_u,t)$ and the same hold for the $H(\tilde{B}_l,t)$.
Now we obtained the isotopy in the definition of the crossing-free.

That completes the proof.
	
\end{proof}
%%%%%%%%%%%%%%%%%%%%%%%%%%%%%
%	 Example
%%%%%%%%%%%%%%%%%%%%%%%%%%%%%
\section{An Example}
In this section, we will show with an example how to construct a pseudo-Anosov map from a crossing-free Hubbard tree.

Recall Example \ref{Example} in Section \ref{Backgraound}, where $f$ is a post-critically finite quadratic polynomial with the characteristic angle $\theta =\frac{1}{5}$.  
Since the Hubbard tree $T$ is crossing-free, then by Theorem \ref{main-theorem} it is extendable to a generalized pseudo-Anosov homeomorphism.

Now using the Markov matrix, we will construct the strips as follows:
 The leading eigenvalue of $$M= \begin{bmatrix}
 0 & 1 & 0 & 0 \\
 0 & 0 & 1 & 0 \\
 1 & 0 & 0 & 1 \\
 1 & 1 & 0 & 0 	
 \end{bmatrix}$$ is $\lambda = 1.39534$ with eigenvector $(0.582522,0.812815,1.13415,1).$ We also have the leading eigenvalue of the transpose $$M^\intercal= \begin{bmatrix}
 0 & 0 & 1 & 1 \\
 1 & 0 & 0 & 1 \\
 0 & 1 & 0 & 0 \\
 0 & 0 & 1 & 0 	
 \end{bmatrix}$$ is $\lambda = 1.39534$ with eigenvector $(1.71667,1.94697,1.39534,1).$
 Normalizing these right- and left-eigenvectors we get:
 $$x=(0.16504438,0.23029267,0.32133565,0.28332729)$$ and $$y=(0.28332657,0.32133626,0.23029289,0.16504428).$$
Now we construct the rectangular strips of dimensions $x_i \times y_i$ as in the following Figure \ref{Const-Rectangles}:
 \begin{figure}[h]
\centering
\begin{minipage}[b]{0.4\linewidth}
\includegraphics[width=1\textwidth]{Thick-Tree18}
\label{fig:minipage2}
\end{minipage}
\quad
\begin{minipage}[b]{0.4\linewidth}
\includegraphics[width=1\textwidth]{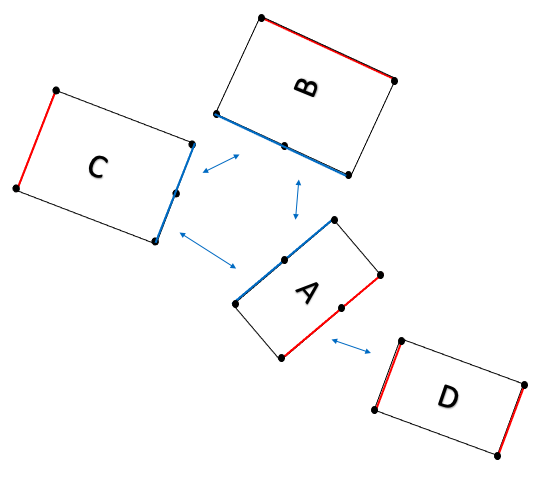}
\label{fig:minipage1}
\end{minipage}
\caption{Constructing Rectangular Strips.}
\label{Const-Rectangles}
\end{figure}  
 
%%%%%%%%%%%%%%%%%%%%%%%%%%%%%
%	 Rectangle Decomposition of a Surface
%%%%%%%%%%%%%%%%%%%%%%%%%%%%%
\subsection{Rectangle Decomposition of a Surface}
 After identifying the sides and applying the map we get the following:
 \begin{figure}[h]
     \centering
     \begin{subfigure}[b]{0.45\textwidth}
         \centering
         \includegraphics[width=\textwidth]{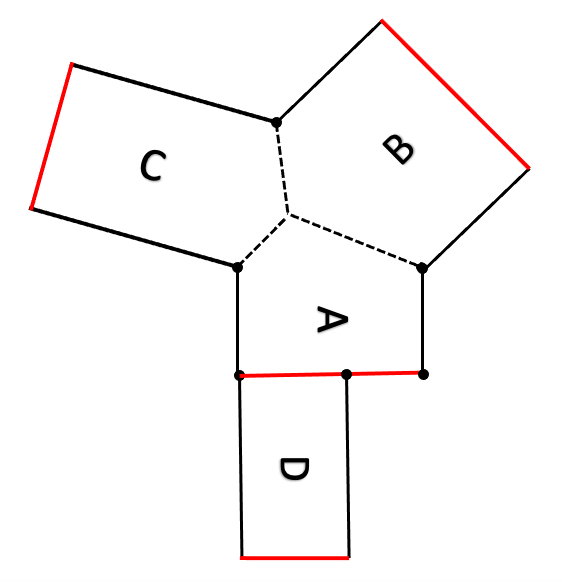}
         \caption*{The rectangle decomposition of a surface from the train track.}
         \label{rectangle-decomp}
     \end{subfigure}
     \hfill
     \begin{subfigure}[b]{0.45\textwidth}
         \centering
         \includegraphics[width=\textwidth]{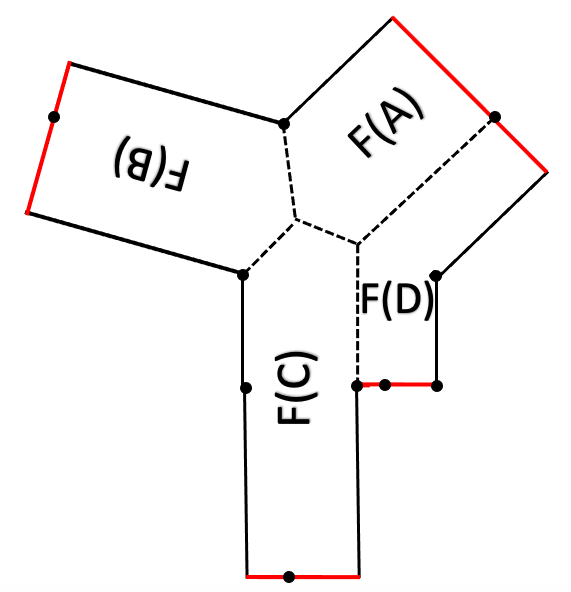}
         \caption*{The rectangle decomposition of the surface after applying the map.}
         \label{Decomp-aftermap}
     \end{subfigure}
     \caption{}
        \label{Decomposition-Surface}
    \end{figure}

 Using the $x$-dimensions of $A,B,C,$ and $D$ from the picture in Figure \ref{Decomposition-Surface}, we will show that the identifications presented in Figure \ref{Strip-final} is correct.
 Since $C$ is mapped to $A$ and $D$, and from Figure 9 we find that there will be an identification between $A$ and an equal part of $D$. We represented this as red arrows left of $C$ in Figure \ref{Strip-final}.
 The remainder of $D$ that is $D-A$ from Figure \ref{Decomposition-Surface}, will be identified with an equal part of $C$. These are shown as yellow arrows on top of $C$ and $B$ in Figure \ref{Strip-final}.
 In the following calculations we refer to the parts from Figure \ref{Decomposition-Surface} and the identifications in Figure \ref{Strip-final}:
   \begin{figure}[h]
\centering
 \includegraphics[scale=.3]{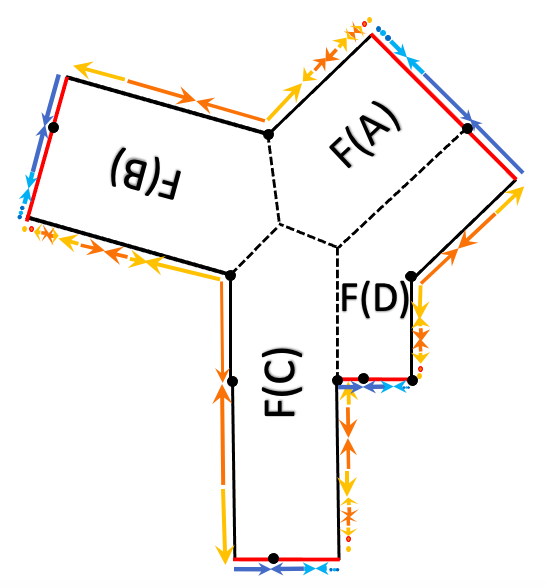}
 \caption{Constructing Rectangular Strips.}
 \label{Strip-final}
 \end{figure} 
 \begin{align*}
 C = & (D-A) + F^{-4}(A+D+(D-A)) +	F^{-4}(F^{-4}(A+D+(D-A)))+ \dots \\
 	 	= & D-A +  F^{-4}(2D)+ F^{-4}( F^{-4}(2D))+ \dots \\
 	 	= & D-A + 2D \sum_{n=1}^{\infty}{(\frac{1}{\lambda ^4})^n}\\
 	= & D-A + 2D (\frac{1}{1-\frac{1}{\lambda ^4}}-1 )\\
 	= & D-(\lambda C-D) + 2D (\frac{ 1} {\lambda ^4 -1}) \\
 	= & -\lambda C + 2D (\frac{ 1} {\lambda ^4 -1}+1) \\
 	= & -\lambda C + 2D (\frac{ \lambda ^4} {\lambda ^4 -1}) \\
 	= & \frac{-\lambda C (\lambda ^4 -1) + 2D \lambda ^4} {\lambda ^4 -1} \\
 	= & \frac{-\lambda C (\lambda ^4 -1) + 2(\lambda C-A) \lambda ^4} {\lambda ^4 -1} \\
 	= & \frac{-\lambda C (\lambda ^4 -2 \lambda +1)}{\lambda ^4 -1} \\
 \end{align*}
 Notice that with $\lambda= 1.39534,$ we have $$\frac{-\lambda C (\lambda ^4 -2 \lambda +1)}{\lambda ^4 -1} \approx C$$

 Similarly, we find that all the identifications in Figure \ref{Strip-final} match the lengths of the $x$-coordinates of the strips.
 
After all the identifications, the result is a surface where all the vertices of the thick tree are identified to one point. Hence the surface is a sphere with countably many singularities accumulating at this one point.


\begin{thebibliography}{1000000}

\bibitem[BLMW]{BLMW} James Belk, Justin Lanier, Dan Marglit, and Rebecca R. Winarski, \textit{Recognizing topological polynomials by lifting trees}, Duke Math. J. 171 (2022), no. 17, 3401-3480.

\bibitem [dC] {dC} Andre de Carvalho, \textit{Extensions, quotients and generalized pseudo-Anosov maps}, In Graphs and patterns in mathematics and theoretical physics, volume 73 of Proc. Sympos. Pure Math., pages 315 {338. Amer. Math. Soc., Providence, RI, 2005.}

\bibitem[dCH] {dCH} Andre de Carvalho and Toby Hall, \textit{Unimodal generalized pseudo-Anosov maps}, Geom. Topol., 8:1127 {1188, 2004.}

\bibitem[DH] {DH} Adrien Douady, John H. Hubbard, \textit{Exploring the Mandelbrot set}, The Orsay Notes.

\bibitem[FM] {FM} Benson Farb, Dan Margalit, \textit{A primer on mapping class groups}, 2012 by Princeton University Press, ISBN 978-0-691-14794-9.

\bibitem[Fa] {Fa} Ethan Farber, \textit{Constructing pseudo-Anosovs from expanding interval maps}, arXiv:2101.01721v3  [math.DS]  17 Feb 2022.

\bibitem[Fr] {Fr}  David Fried, \textit{Growth rate of surface homeomorphisms and flow equivalence}, Ergodic Theory Dynam. Systems, 5(4):539 {563, 1985.}

\bibitem[Ha] {Ha} Toby Hall, \textit{The creation of horseshoes}, Nonlinearity, 7(3):861 {924, 1994.} 

\bibitem[Mi] {Mi} John Milnor, \textit{Dynamics in One Complex Variable}, Introductory Lectures (Partially revised version of 9-5-91) Institute for Mathematical Sciences, SUNY, Stony Brook, NY, 1990. 

\bibitem[Mi] {Mi} John Milnor, \textit{Periodic Orbits, External Rays and the Mandelbrot set: an expository account}, 2000, Geometrie complexes at systemes dynamiques (Orsay, 1995), pp. xiii, 277-333.

\bibitem[Th] {Th} William P. Thurston, \textit{Entropy in dimension one}, In Frontiers in complex dynamics, volume 51 of Princeton Math. Ser., pages 339 {384. Princeton Univ. Press, Princeton, NJ, 2014.}

\bibitem[Ti] {Ti} Giulio Tiozzo, \textit{Continuity of core entropy of quadratic polynomials}, Invent. Math. 203 (2016), no. 3,8 91-921.

\end{thebibliography}
\end{document}